\documentclass[11pt]{amsart}
\usepackage[T1]{fontenc}
\usepackage{a4wide}
\usepackage{mathpazo}
\usepackage{mathrsfs}
\usepackage{amssymb, amsmath}
\usepackage{enumitem}
\usepackage{tikz}

\newcommand{\sC}{\mathscr C}
\newcommand{\sM}{\mathscr M}

\newcommand{\ZZ}{\mathbf{Z}}
\newcommand{\RR}{\mathbf{R}}

\newcommand{\Span}{\mathrm{Span}}
\newcommand{\Sym}{\mathrm{Sym}_\infty}
\newcommand{\Id}{\mathrm{Id}}

\newcommand{\one}{\boldsymbol{1}}
\newcommand{\se}{\subseteq}

\newcommand{\fhi}{\varphi}
\newcommand{\ro}{\varrho}

\newcommand{\inv}{^{-1}}
\newcommand{\vertiii}[1]{{\left\vert\kern-0.25ex\left\vert\kern-0.25ex\left\vert #1 
    \right\vert\kern-0.25ex\right\vert\kern-0.25ex\right\vert}}
\theoremstyle{plain}
\newtheorem{thm}{Theorem}
\newtheorem*{thm*}{Theorem}
\newtheorem{lem}[thm]{Lemma}
\newtheorem{prop}[thm]{Proposition}
\newtheorem{cor}[thm]{Corollary}
\theoremstyle{definition}
\newtheorem*{defn*}{Definition}
\newtheorem{defn}[thm]{Definition}
\newtheorem{example}[thm]{Example}
\newtheorem*{example*}{Example}
\newtheorem{rem}[thm]{Remark}

\newtheorem{scholium}[thm]{Scholium}
\newtheorem{question}[thm]{Question}
\newtheorem{problem}[thm]{Problem}
\newtheorem*{rem*}{Remark}
\begin{document}

\title[Fixed points in convex cones]{Fixed points in convex cones}

\begin{abstract}
We propose a fixed-point property for group actions on cones in topological vector spaces. In the special case of equicontinuous representations, we prove that this property always holds; this statement extends the classical Ryll-Nardzewski theorem for Banach spaces. When restricting to cones that are locally compact in the weak topology, we prove that the property holds for all distal actions, thus extending the general Ryll-Nardzewski theorem for all locally convex spaces.

Returning to arbitrary actions, the proposed fixed-point property becomes a group property, considerably stronger than amenability. Equivalent formulations are established and a number of closure properties are proved for the class of groups with the fixed-point property for cones.
\end{abstract}
\author{Nicolas Monod}

\address{EPFL, 1015 Lausanne, Switzerland}
%
%
%
\maketitle

\section{Introduction and results}
\begin{flushright}
\begin{minipage}[t]{0.7\linewidth}\itshape\small
As an art, mathematics has certain fundamental techniques.\\
In particular, the area of fixed-point theory has its specific artistry.
\begin{flushright}
\upshape\small
Sister JoAnn Louise Mark~\cite[p.~iii]{Mark_PhD}
\end{flushright}
\end{minipage}
\end{flushright}
\vspace{1mm}
\begin{flushright}
\begin{minipage}[t]{0.7\linewidth}\itshape\small
Traiter la nature par le cylindre, la sph\`ere, le c\^one\ \ldots
\begin{flushright}
\upshape\small
Paul C\'ezanne, letter to \'Emile Bernard, 15 April 1904~\cite[p.~617]{Bernard07}
\end{flushright}
\end{minipage}
\end{flushright}
\subsection{About fixed points}
A large part of fixed-point theory for groups, and perhaps the most widely applied, takes place in the following setting:

\smallskip
\begin{center}
A group $G$ acts on a compact convex set $K\neq \varnothing$ in a topological vector space $E$.
\end{center}

\smallskip
We will agree that the action on $K$ comes from a continuous linear representation on $E$; even if we are only given a continuous affine action on $K$, a suitable $G$-space $E$ can be reconstructed (as dual of the space of affine continuous functions on $K$).

\medskip
The following two themes are of particular interest:

\begin{enumerate}[label=(\Roman*)]
\item General fixed-point theorems without any restriction on the group $G$; this requires assumptions on $K$, $E$ and/or the nature of the action.\label{pt:RN}
\item Make no assumptions on $K$ nor on the action. Then, having always a $G$-fixed point becomes a striking property of a group $G$. \label{pt:amen}
\end{enumerate}

A powerful example of~\ref{pt:RN} is the original Ryll-Nardzewski theorem~\cite{Ryll-Nardzewski62}. Here the assumptions are that $E$ is a Banach space endowed with its \emph{weak} topology and that the action is (norm-)isometric. It is because compactness is only assumed in the weak topology that this result is so much stronger than earlier theorems of Kakutani~\cite{Kakutani38} and Hahn~\cite{Hahn67}; Ryll-Nardzewski further generalised his statement to distal actions in general locally convex spaces~\cite{Ryll-Nardzewski} (see also~\cite{Glasner-Megrelishvili} for further extensions).

\medskip
As for~\ref{pt:amen}, this fixed-point property (considered e.g.\ by Furstenberg~\cite[Def.~1.4]{Furstenberg63}) turns out to be equivalent to \emph{amenability}~\cite{Day61}, \cite{Rickert67}. Although amenability was originally introduced by von Neumann~\cite{vonNeumann29S} in his study of Banach--Tarski paradoxes, it has since become a familiar concept in a number of unrelated areas. The connexion to~\ref{pt:amen} is that finitely additive measures normalised to have \emph{total mass one} form a compact convex set.

\subsection{Compact is too restrictive}
The context above is too limited for certain applications. Even in the original case of Banach--Tarski paradoxes, the key question is whether our space $\RR^3$ carries a finitely additive measure defined on all subsets, invariant under isometries, \itshape and normalised for a given subset $A\se \RR^3$ of interest\upshape, e.g.\ a ball. In particular, such a measure $\mu$ would generally need to be unbounded (and infinite on large subsets). A natural invariant normalisation condition is $0<\mu(gA)<\infty$ for all isometries $g$; such measures form an invariant convex cone.

For the purpose of giving examples of paradoxes, Hausdorff~\cite{Hausdorff14_article} and Banach--Tarski~\cite{Banach-Tarski} could bypass this issue because they could consider instead subsets of the group of rotations; there, measures can be normalised over the entire group, thus still forming a compact set.

However, the question considered by von Neumann~\cite[p.~78]{vonNeumann29S} is really whether there exists a finitely additive measure on a set $X$, invariant under a group $G$, and normalised for some subset $A\se X$ that need not be $X$ -- nor indeed be invariant. In fact Tarski proved that this is equivalent to the non-paradoxicality of $A$~\cite[end of~\S\,3]{Tarski38}. Groups that always admit such a measure are called \textbf{supramenable} since Rosenblatt~\cite[4.5]{Rosenblatt_PhD}.

\medskip
It is known that supramenability is much stronger than amenability. It fails to hold for metabelian groups such as the Baumslag--Solitar group $\mathrm{BS}(1,2)$ or the lamplighter group $\ZZ/2 \wr \ZZ$ because supramenable groups cannot contain non-abelian free semigroups~\cite[2.3]{Rosenblatt_PhD}. In fact, supramenability is \emph{equivalent} to not containing a Lipschitz-embedded binary tree~\cite[Prop.~3.4]{Kellerhals-Monod-Rordam}.

On the other hand, examples of supramenable groups include all abelian groups, all locally finite groups, and more generally the groups that are \textbf{subexponential} in the sense that for every finite subset $S$ of the group, $\lim_{n\to\infty} |S^n|^{1/n}=1$, see~\cite[3.5]{Rosenblatt_PhD}. This argument was in fact already used in 1946 by Sierpi{\'n}ski, see pp.~115--117 in~\cite{Sierpinski46}.

\medskip
Another situation where compact convex sets are too restrictive is the case of Radon measures on locally compact spaces. Amenability can be characterised in terms of the existence of invariant probability measures on \emph{compact} topological spaces, but the non-compact case is of course also very interesting (we think e.g.\ of Haar measures); it has been explored e.g.\ in~\cite{Kellerhals-Monod-Rordam}, \cite{Matui-Rordam}. The positive Radon measures form, again, a convex cone; it has in general no affine base. A basic example is the affine group of the line, which is amenable but does not fix any non-zero Radon measure on the line.

\smallskip
Many other examples of convex cones without base arise, for instance in non-commutative operator algebras, see e.g.~\cite{Elliott-Robert-Santiago}.

\subsection{Convex cones}
Our starting point will be the following, to be further refined below (while earlier work is briefly reviewed in Section~\ref{sec:other}).

\smallskip
\begin{center}
A group $G$ acts on a convex proper cone $C \se E$ in a topological vector space $E$.\\
The question is whether there exists a \emph{non-zero} fixed point.
\end{center}

\smallskip
Recall that \emph{proper} means $C \cap (-C) = \{0\}$, or equivalently that the associated pre-order on $E$ is an order. It is understood that a \textbf{representation} of a group on a vector space $E$ is an action by linear maps that are assumed \emph{continuous} (hence weakly continuous) if a topology is given on $E$ and \emph{order-preserving} if a vector space pre-order (equivalently, a convex cone) is given for $E$. We then also speak of a representation on the cone, but the ambient linear representation is understood as part of the data.

\medskip
The topological condition on $C$ shall be that $C$ is \emph{weakly complete}. Choquet has eloquently made the case in his ICM address~\cite{Choquet63} that this condition captures many familiar cones encountered in analysis. Oftentimes, the spaces $E$ that we need to consider are weak in the sense that their topology coincides with the associated weak topology; in any case, the weak completeness assumption justifies that we shall only consider locally convex Hausdorff topological vector spaces (as is generally done in the compact convex case too).

For instance, the cone of positive Radon measures on a locally compact space is complete for the vague topology (Prop.~14 in~\cite[III \S\,1 No.\,9]{BourbakiINT14}), hence weakly complete since this topology is weak.

\medskip
This setting obviously includes the earlier picture, for if $K\se E$ is convex and compact or weakly compact, then the cone
$$C= K\rtimes \RR_+ :=\big\{(tk, t): t\geq 0,\, k\in K\big\}\kern10mm \text{in }\ E \oplus \RR$$
is convex, proper and weakly complete -- indeed even weakly locally compact~\cite[II \S\,7 No.\,3]{BourbakiEVT81}. This special case admits a base, $K\times\{1\}$, and this base is even $G$-invariant.

\subsection{How groups act on cones}
Having given up compactness, we need some bounds on the actions on cones; otherwise, even the group $\ZZ$ will act without non-zero fixed point on \emph{any} cone, simply by non-trivial scalar multiplication. We introduce two conditions -- which both trivially hold in the above case $C= K\rtimes \RR_+ $.

\begin{defn}
A representation on a convex cone $C$ in a topological vector space is of \textbf{locally bounded type} if there exists an element $x\neq 0$ in $C$ whose orbit is bounded. Here \emph{bounded} is understood is the sense of topological vector spaces rather than orders, that is: absorbed by any neighbourhood of zero~\cite[III \S\,1 No.\,2]{BourbakiEVT81}.
\end{defn}

This condition holds in all the examples that we aim at. For instance, if $G$ acts on a locally compact space $X$ and $C$ is the cone $C=\sM_+(X)$ of positive Radon measures, then the $G$-orbit of any Dirac mass is bounded. Indeed, the closure of the image of $X$ under the canonical map $X\to \sM_+(X)$ realises the one-point compactification of $X$ (Prop.~13 in~\cite[III \S\,1 No.\,9]{BourbakiINT14}). Notice that typically the orbits admit $0$ as an accumulation point.

In this example with $C=\sM_+(X)$, the convex sub-cone of bounded orbits is in fact \emph{dense} in $C$~\cite[III \S\,2 No.\,4]{BourbakiEVT81}, a situation that we can often reduce to (Proposition~\ref{prop:min} below). However, it is critical to keep in mind that even then, there are usually also unbounded orbits. If for instance we let $G=\ZZ$ act by translation on $C=\sM_+(\ZZ) \cong (\RR_+)^\ZZ$, then both the bounded and the unbounded orbits are dense; there are even invariant rays where $\ZZ$ acts unboundedly, namely exponentials.

\bigskip
We turn to the second condition, dual to local boundedness. As a motivation, we recall that \emph{any} infinite group acts on some locally compact space without preserving any non-zero Radon measure~\cite[4.3]{Matui-Rordam}. By contrast, it was shown in~\cite{Kellerhals-Monod-Rordam} that we can characterise supramenability by restricting our attention to \emph{cocompact} actions on locally compact spaces.

It is easily verified that for this specific example of Radon measures in the vague topology, the following general definition is precisely equivalent to the cocompactness of the $G$-action on the underlying space $X$:

\begin{defn}
A representation of a group $G$ on an ordered topological vector space $E$ is of \textbf{cobounded type} if there exists a continuous linear form $\fhi\in E'$ which $G$-dominates all other $\lambda\in E'$.

Here, an element $\fhi$ in an arbitrary (pre-)ordered vector $G$-space is said to \textbf{$G$-dominate} $\lambda$ if there exists $g_1, \ldots, g_n\in G$ with $\pm\lambda \leq \sum_i g_i \fhi$. The (topological) dual $E'$ is always endowed with the corresponding dual $G$-representation and dual pre-order.
\end{defn}

\begin{rem}\label{rem:cobounded}
As a further motivation for this definition, another example of a weakly complete cone with a $G$-representation (of locally bounded type) without non-zero fixed point is the cone of positive elements in $\ell^2(G)$. This cone is weakly complete~\cite[Prop.~5]{Becker85}. However, there is no non-zero fixed point as soon as $G$ is infinite.
\end{rem}

Finally, we can formalise our group-theoretical property:

\begin{defn}
A group $G$ has \textbf{the fixed-point property for cones} if every $G$-representation of locally bounded cobounded type on any weakly complete proper convex cone has a non-zero fixed point.
\end{defn}

In the light of theme~\ref{pt:amen} outlined at the beginning of this introduction, this is a strengthening of amenability. As we shall see, all subexponential groups enjoy the fixed-point property for cones (Theorem~\ref{thm:stable}\ref{pt:stable:SE}). On the other hand, this property implies supramenability and therefore the metabelian groups indicated above do not have the fixed-point property for cones.

\subsection{Groups with the fixed-point property for cones}
Just as the fixed-point property defined by Furstenberg turned out to be one of several equivalent incarnations of amenability, we shall establish a few characterisations in the conical setting. These equivalences (Theorem~\ref{thm:equiv}) will facilitate our further investigation of the fixed-point property.

\medskip
Starting from our original motivation, we would like to obtain invariant finitely additive measures normalised for arbitrary subsets of a $G$-space $X$; the key case concerns subsets of $X=G$ itself. In fact, we can more generally consider invariant ``integrals'' normalised for functions rather than subsets. To this end, given a bounded function $f\geq 0$ on $G$, we denote by $\ell^\infty(G; f)$ the ordered vector $G$-space of functions $G$-dominated by $f$.

\begin{defn}\label{def:int}
Let $f\geq 0$ be a non-zero bounded function on $G$. An \textbf{integral on $G$ normalised for $f$} is a $G$-invariant positive linear form $\mu$ on the vector space $\ell^\infty(G; f)$ such that $\mu(f)=1$.
\end{defn}

It should not be surprising that such integrals can be obtained as fixed points in cones. The existence of integrals under the stronger assumption that the group is subexponential was established in~\cite [Cor.~3]{Jenkins76} and~\cite[2.16]{Kellerhals_PhD}.

\medskip
A simpler task would be to find a functional on the smaller space $\Span(G f)$ spanned by the $G$-orbit of $f$. However, even in the case of a subset $A\se X$ with $f=\one_A$, it is not clear a priori that such a condition should imply supramenability. This amounts indeed to the question asked in Greenleaf's classical 1969 monograph~\cite[p.~18]{GreenleafBook} (and met with skepticism in~\cite{Jenkins80}, Remark p.~369).

We shall give a positive answer to this question in Corollary~\ref{cor:TP-SM} below. For now, we recall that there is a $G$-invariant positive linear form on $\Span(G f)$ normalised for $f$ if and only if the following \emph{translate property} introduced in~\cite{Rosenblatt_PhD} holds for $f$.

\begin{defn}\label{def:TP}
Let $f\geq 0$ be a non-zero bounded function on $G$. Then $f$ satisfies the \textbf{translate property} if
$$\sum_{i=1}^n  t_i g_i f \geq 0 \ \Longrightarrow\  \sum_{i=1}^n  t_i \geq 0\kern10mm(\forall t_i \in \RR, g_i\in G).$$
The group $G$ is said to have the translate property if this holds for all $f$.
\end{defn}

This condition is equivalent to the existence of an integral on $\Span(G f)$ because it ensures at once that the linear form $\sum_i  t_i g_i f \mapsto \sum_i  t_i$ is both well-defined and positive. It has been known~\cite[\S\,1]{GreenleafBook}, \cite[\S\,1]{Rosenblatt_PhD} that such an integral can be extended to all of $\ell^\infty(G; f)$ \emph{provided that $G$ is amenable}. Therefore, we shall need to show that the translate property implies amenability.

\medskip
Von Neumann's investigation of invariant finitely additive measures was not restricted to subsets of $G$ acting on itself, but it can readily be reduced  to it. Likewise, we shall extend the translate property to abstract ordered vector spaces and consider the \emph{abstract translate property} (defined in Section~\ref{sec:ATP} below). Finally, in parallel to the finitary characterisations of amenability in terms of Reiter properties, we consider in Section~\ref{sec:Reiter} the \emph{$f$-Reiter property} corresponding to a function $f$ on $G$.

\begin{thm}\label{thm:equiv}
The following properties of a group $G$ are equivalent.
\begin{enumerate}[label=(\arabic*)]
\item The fixed-point property for cones.\label{pt:equiv:FPC}
\item The translate property.\label{pt:equiv:TP}
\item The abstract translate property.\label{pt:equiv:ATP}
\item For every non-zero bounded function $f\geq 0$ on $G$, there is an integral on $G$ normalised for $f$.\label{pt:equiv:MG}
\item $G$ satisfies the $f$-Reiter property for every $0\neq f\in \ell^\infty(G)$.\label{pt:equiv:Reiter}
\end{enumerate}
\end{thm}

We now investigate the class of all groups with the fixed-point property for cones, and its stability under basic group constructions. The case of metabelian groups without this property alerts us to the fact that group extensions will not, in general, preserve the fixed-point property for cones. In fact, we do not even know whether the Cartesian product (of two groups) preserves it. The same question is open for supramenable groups and was first asked in the 1970s~\cite[p.~49]{Rosenblatt_PhD}, \cite[p.~51]{Rosenblatt74}.

\smallskip
The closure under quotients is obvious, but for the other hereditary properties it will be helpful to be armed with the various characterisations of Theorem~\ref{thm:equiv}. We find our proof of point~\ref{pt:stable:central} laborious; perhaps the reader will devise a simpler argument.

\begin{thm}\label{thm:stable}
The class of groups with the fixed-point property for cones is closed under:
\begin{enumerate}[label=(\arabic*)]\setcounter{enumi}{-1}
\item quotients,
\item subgroups,\label{pt:stable:sub}
\item directed unions,\label{pt:stable:union}
\item Cartesian products with subexponential groups,\label{pt:stable:SE}
\item central extensions,\label{pt:stable:central}
\item extensions by and of finite groups.
\end{enumerate}
\end{thm}

\subsection{Two conical fixed-point theorems}
We now return to the first theme~\ref{pt:RN}, namely fixed-point theorems without assumptions on the group~$G$. Our first fixed-point statement remains in the generality of weakly complete cones $C\se E$ as above but under the assumption that the $G$-representation is \emph{equicontinuous} on $E$. In that case, the local boundedness condition is automatically satisfied (if $C\neq \{0\}$).

In particular, the following contains the Ryll-Nardzewski theorem for isometries of Banach spaces as the special case of a cone $C= K\rtimes \RR_+$.

\begin{thm}\label{thm:RNC}
Let $G$ be a countable group with a representation of cobounded type on a weakly complete convex proper cone $C\neq \{0\}$ in a (Hausdorff) locally convex space $E$.

If the representation is equicontinuous on $E$, then $G$ fixes a non-zero point in $C$.
\end{thm}

\begin{rem}
The most general form of the Ryll-Nardzewski theorem considers actions that are \emph{distal} on a compact convex set. This is a condition only on this subset of $E$, and moreover it is a weaker condition than equicontinuity (compare~\cite{Glasner87} and references therein). In parallel to this generality, Theorem~\ref{thm:RNCD} below will address actions that are only assumed distal on a cone. At the opposite spectrum of generality, Remark~\ref{rem:Ban} proposes a shorter proof for isometric representations on Banach spaces.
\end{rem}

The reader will have noticed that one inconspicuous assumption on $G$ \emph{did} sneak into Theorem~\ref{thm:RNC}, namely that $G$ be countable. In the classical compact case, this restriction can immediately be lifted by a compactness argument. For the above statement, however, it is needed even in the most familiar case of Hilbert spaces:

\begin{example}\label{ex:L2}
Let $V=L^2([0,1])$ and let $C$ be the convex proper cone of function classes that are a.e.\ non-negative. For the usual scalar product of $V$, the cone $C$ is self-dual. Since moreover $V=C-C$, it follows that $C$ is weakly complete, see e.g.~\cite[Prop.~5]{Becker85}. (As was pointed out to us by de la Harpe, a nice abstract characterisation of this cone was obtained by Sz.-Nagy~\cite{Sz-Nagy50}.)

Let further $G$ be the group of non-singular automorphisms of the measure space $[0,1]$. The usual isometric $G$-representation on $V$ defined a.e.\ by
$$(gf)(x) = \left( \frac{dg \mu}{d\mu}(x)\right)^{\frac12} f(g\inv x) \kern 10mm (g\in G, f\in L^2),$$
where $dg \mu/d\mu \in L^1([0,1])$ denotes the Radon--Nikodym derivative of $g \mu$ with respect to the Lebesgue measure $\mu$, preserves the cone $C$.

This $G$-action has no non-zero fixed point in $V$ but we claim that it is of cobounded type. In fact, we claim that the unit constant function $\one\in C$, viewed as an element of the dual, is $G$-dominating. Let indeed $f\in C$ be arbitrary and define $f_1=f\vee \one \geq f$ to be the maximum of $f$ and $\one$. It suffices to show that for some $g\in G$ we have $g\inv(\|f_1\|  \one) =  f_1$. The assignment $g(x) =\|f_1\| ^{-2}  \int_0^x f_1^2 d\mu$ defines an increasing map $[0,1]\to [0,1]$ which yields a non-singular automorphism since $0<f_1<\infty$. Now $g\inv(\|f_1\|  \one) =  f_1$ holds by construction.
\end{example}

\smallskip
We can nonetheless remove this last hypothesis on $G$ if we change the topological assumption on the cone. In fact, the countability of $G$ is only used in Theorem~\ref{thm:RNC} to establish that $C$ must be weakly locally compact -- this is a consequence of equicontinuity which is in strong contrast to the general cones that we had to deal with for the group-theoretical property.

Once we consider cones that are locally compact in the weak topology, we can dispense of any restriction on $G$ and deal with the distal case. In particular, the following subsumes the general case of the Ryll-Nardzewski theorem as the particular instance of a cone $C= K\rtimes \RR_+$.

\begin{thm}\label{thm:RNCD}
Let $G$ be a group with a representation of locally bounded type on a weakly locally compact convex proper cone $C$ in a (Hausdorff) locally convex space $E$.

If the action is distal on $C$, then $G$ fixes a non-zero point in $C$.
\end{thm}

Notice that now the coboundedness assumption has become redundant for this statement, but the local boundedness assumption came back.

For Theorem~\ref{thm:RNC}, the coboundedness assumption was needed even in the case of isometric actions on Banach spaces: this is illustrated by the example of $\ell^2(G)$ given in Remark~\ref{rem:cobounded}.

\bigskip\noindent
\textbf{Acknowledgements.} I am grateful to Taka Ozawa and to the anonymous referee for helpful remarks on preliminary versions of this article.

\setcounter{tocdepth}{1}
\tableofcontents
%
\section{Proof of the conical fixed point theorems}
We begin with the proof of Theorem~\ref{thm:RNCD} and retain its notation. One input for the proof will be established later (Proposition~\ref{prop:min}) because we shall need it again in a more general setting.

To get started, here is a consequence of the distality of the $G$-action on $C$.

\begin{lem}\label{lem:dist}
Let $\epsilon > 0$ and pick any $x,z\in C$ not contained in a common ray. Suppose that both orbits $Gx$, $Gz$ are bounded and that the $G$-action on $C$ is distal. Then there is a neighbourhood $U$ of the origin in $E$ such that for all $g\in G$ and all $s,t\geq \epsilon$ we have $g(s x - t z) \notin U$.
\end{lem}

\begin{proof}
Suppose for a contradiction that the conclusion does not hold. Then there are nets of elements $g_\alpha\in G$ and $s_\alpha, t_\alpha \geq \epsilon$ (with index $\alpha$ in some directed set) such that $g_\alpha (s_\alpha x - t_\alpha z)$ converges to zero. Upon replacing them by subnets and possibly exchanging $x$ with $z$, we can assume that $s_\alpha \geq t_\alpha$ holds for all $\alpha$. Appealing to a further subnet we can assume that $t_\alpha/s_\alpha$ converges to some $\tau\in[0,1]$. We shall contradict the distality assumption by showing that $g_\alpha (x - \tau z)$ converges to zero, recalling that $x\neq \tau z$. Indeed, write
$$g_\alpha (x - \tau z) \ =\ g_\alpha (x - (t_\alpha/s_\alpha) z) + ( t_\alpha/s_\alpha  - \tau) g_\alpha z.$$
The first term is $s_\alpha\inv g_\alpha (s_\alpha x - t_\alpha z)$, which still converges to zero since $s_\alpha\inv$ is bounded. The second term converges to zero because $g_\alpha z$ is bounded in $E$.
\end{proof}

Since $C$ is weakly locally compact and proper, it admits a weakly compact base $K_1$, see Ex.~21 in~\cite[II \S\,7]{BourbakiEVT81} (this result is due to Dieudonn\'e and Klee~\cite[2.4]{Klee55c}). \itshape We emphasise that this base need not be invariant\upshape. This means that we are not in the classical setting, but rather in a ``projective'' context; the distance between these two settings can be appreciated e.g.\ in light of the \emph{Tychonoff property} recalled in Section~\ref{sec:other}.  Nevertheless, we will be able to use suitably adapted extensions of some of the known arguments in proofs of the Ryll-Nardzewski theorem.

\medskip
Applying  a compactness argument to the intersections of $K_1$ with sub-cones of fixed points in $C$, we see that it suffices to find a non-zero fixed point for each countable subgroup of $G$; we can thus assume that $G$ is countable.

We claim that we can furthermore suppose that $C$ is minimal amongst all those invariant closed convex sub-cones of $C$ that contain a non-zero bounded orbit. This is more delicate since the set of bounded orbits is a priori not closed, but we establish this fact in the more general weakly complete setting in Proposition~\ref{prop:min} below. We point out that here $C$ is trivially of cobounded type as assumed in  Proposition~\ref{prop:min} because it admits a base. Notice that in particular we have reduced to the case where $C$, and hence also $K_1$, is separable.

\medskip
We now proceed to show that $K_1$ contains only one point whose orbit in $C$ is bounded. Since the corresponding ray will then be $G$-invariant by uniqueness, it will in fact be point-wise fixed by distality. Therefore, this will then finish the proof.

\medskip
Suppose for a contradiction that $K_1$ contains distinct points $x,y$ with bounded orbits and denote their midpoint by $z$. In particular the orbit of $z$ is also bounded. Let $\psi\in E'$ be the linear form corresponding to the compact base $K_1$, see~\cite[II \S\,7 No.\,3]{BourbakiEVT81}. By boundedness of the chosen orbits, there is $\epsilon >0$ such that $\psi(g x)$, $\psi(g y)$ and $\psi(g z)$ are bounded by $1/\epsilon$ for all $g\in G$. We apply Lemma~\ref{lem:dist} to both pairs $x,z$ and $y,z$ with this $\epsilon$. Since $E$ is locally convex, the lemma implies that there exists a continuous seminorm $q$ on $E$ and $\delta>0$ such that
$$q\big(g(s x - t z)\big)\geq \delta, \kern5mm q\big(g(s y - t z)\big)\geq \delta\kern10mm\forall\,g\in G,\ \forall\,s,t\geq \epsilon.$$
We consider the convex weakly compact set $K_0=K_1-u\se\ker(\psi)$ corresponding to $K_1$ under translation by some $u\in K_1$. Applying to $K_0$ a denting theorem such as the lemma p.~443 in~\cite{Asplund-Namioka}, we find a continuous linear form $\lambda_0\in \ker(\psi)'$ and a constant $\alpha\in\RR$ such that the set $U_0=\{c\in K_0: \lambda_0(c) < \alpha \}$ is non-empty but of $q$-diameter~$<\delta$. Denote by $U_1\se K_1$ the corresponding set $U_1=U_0+u$, which has the same $q$-diameter.

\medskip
We choose a pre-image $\lambda_1\in E'$ under the map $E'\to E'/ \RR\psi \cong \ker(\psi)'$ and define $\lambda\in E'$ by $\lambda=\lambda_1 -(\lambda_1(u)+\alpha) \psi$. Notice that we have
$$U_1 = K_1 \cap \{c\in C : \lambda(c) < 0\}.$$
In particular, by minimality of $C$ in the class of closed invariant sub-cones containing bounded orbits, the bounded orbit $Gz$ cannot remain in the closed convex sub-cone $\{c\in C : \lambda(c) \geq 0\}$ of $C$. We can thus choose $g\in G$ with $\lambda (g z) <0$. At least one of $\lambda (g x)$ or $\lambda (g y)$ must also be negative; we assume $\lambda (g x)<0$. We write $s=1/\psi(gx)$ and $t=1/\psi(gz)$, recalling that $\psi$ does not vanish on $C\setminus\{0\}$. Now $sg x=g sx$ belongs to $U_1$ and likewise for $g tz$. This implies $q(g (sx-tz))<\delta$, which is a contradiction since our choice guarantees $s,t\geq \epsilon$. This completes the proof of Theorem~\ref{thm:RNCD}.\qed

\bigskip
We now turn to Theorem~\ref{thm:RNC}. Consider an enumeration $\{g_n\}_{n=1}^\infty$ of $G$ and let $\fhi\in E'$ be $G$-dominating for $E'$. We endow $E'$ with the weak-* topology induced by $E$. By our assumption on the action, the set $\{g_n \fhi\}$ is equicontinuous (on $E$); we use throughout that equicontinuity implies in fact uniform equicontinuity for families of linear maps (Prop.~5 in~\cite[X \S\,2 No.\,2]{BourbakiTGII_alt}). Therefore, this set is relatively compact in $E'$, see Cor.~2 in~\cite[III \S\,3 No.\,4]{BourbakiEVT81}). The same holds for the convex balanced hull of $\{g_n \fhi\}$, since the convex balanced hull of an equicontinuous set is equicontinuous. Note that we argue by equicontinuity since, in general, the closed convex hull of a compact set need not be compact (in our case it would be true if $E$ was supposed \emph{barrelled}, but not in general). It follows that we can define an element $\psi\in E'$ by $\sum_{n=1}^\infty 2^{-n} g_n\fhi$ since the sequence of partial sums is weak-* Cauchy in this convex balanced hull.

\smallskip
Since $\fhi$ is $G$-dominating, $\psi$ is an order-unit for $E'$, i.e.\ for any $\lambda\in E'$ there is $t\in \RR_+$ with $\lambda \leq t \psi$. Thus, $\psi$ is strictly positive on $C\setminus \{0\}$; we define $K_1$ to be the intersection of $C$ with the affine hyperplane $\psi\equiv 1$.

\medskip
We claim that $K_1$ is weakly compact. Since $\psi$ is an order-unit, every $\lambda\in E'$ is bounded on $K_1$. Thus $K_1$ is weakly bounded. It is also weakly complete since $C$ is so, and therefore weak compactness follows, see e.g.~\cite[23.11]{Choquet_LAII}.

\medskip
At this point $C$ has been shown to have a weakly compact base $K_1$, so that it is locally compact in the weak topology. Since the representation is equicontinuous, all orbits are bounded and the action is distal. In other words, we can apply Theorem~\ref{thm:RNCD} to finish the proof of Theorem~\ref{thm:RNC}.\qed

\begin{rem}\label{rem:Ban}
In the special case where $E$ is a Banach space (with isometric representation), we can instead reduce ourselves to a double application of the Ryll-Nardzewski theorem thanks to Davis--Figiel--Johnson--Pe{\l}czy{\'n}ski renorming~\cite{Davis-Figiel-Johnson-Pelczynski}. Indeed, the intersection $L$ of $C$ with the unit ball $E_1$ of $E$ is weakly compact because it is weakly complete and bounded; therefore, by the Krein--Smulian theorem, so is the closed convex hull $W$ of $L\cup (-L)$, which is now symmetric. We recall the Davis--Figiel--Johnson--Pe{\l}czy{\'n}ski norm $\vertiii{\cdot}$ associated to $W$. Let $\|\cdot\|_n$ be the gauge norm associated to the set $2^n W+ 2^{-n} E_1$ and define the (possibly infinite) norm $\vertiii{\cdot}=(\sum_{n=1}^\infty \|\cdot\|_n^2)^{1/2}$. Let $F\se E$ be the set of $x$ with $\vertiii{x}$ finite and endow $F$ with the norm $\vertiii{\cdot}$. Then Lemma~1 of~\cite{Davis-Figiel-Johnson-Pelczynski} states that $F$ is a reflexive Banach space. By construction, the inclusion map $F\to E$ is continuous and $F$ contains $W$ in its unit ball; in particular, $C\se F$.

We observe that $G$ preserves $F$ and acts isometrically for its new norm since the sets $2^n W+ 2^{-n} E_1$ are invariant. Moreover, it is shown in~\cite[Thm.~1]{Becker89} that $C$ is also weakly complete for the weak uniform structure of $F$. A weak compactness argument in the reflexive space $F$ shows that $\psi|_{F}$ is still an order-unit for $F'$ and a Baire argument in $F'$ implies that $\psi|_F$ is actually an upper bound for a small ball in $F'$. This implies that the closed convex hull of $G \psi|_F$ does not contain zero. By reflexivity, we can apply Ryll-Nardzewski in $F'$ and find a non-zero positive fixed point in $F'$. This provides an \emph{invariant} base for $C$ and a second application of Ryll-Nardzewski to this base yields a fixed point in $C$.
\end{rem}

\section{Integrals and cones}\label{sec:MG}
We begin the proof of Theorem~\ref{thm:equiv} by showing that the fixed-point property for cones is, as expected, sufficiently general to be applied to cones of integrals in the sense of Definition~\ref{def:int}. In other words, we establish~\ref{pt:equiv:FPC}$\Rightarrow$\ref{pt:equiv:MG}:

\bigskip
Let $f\geq 0$ be a non-zero bounded function on $G$. Consider the algebraic dual $E=\ell^\infty(G;f)^*$ endowed with the weak-* topology defined by $\ell^\infty(G;f)$. The convex cone $C\se E$ of positive linear forms is proper since $\ell^\infty(G;f)$ is spanned by positive elements. Moreover, $C$ is weakly complete  (i.e.\ complete) in this topology (compare~\cite{Choquet62a} or recall that algebraic duals are weak-* complete~\cite[II \S\,6 No.\,7]{BourbakiEVT81}).

To justify that the $G$-action is of locally bounded type, consider an element $g\in G$ with $f(g)\neq 0$. The evaluation at $g$ is a non-zero element $x\in C$ and we claim that its $G$-orbit is bounded in $E$. Indeed, a neighbourhood basis at $0$ in $E$ is given by sets $U=\{y: |y(f_i)| < \epsilon\ \forall i\}$ with $f_1, \ldots, f_n\in \ell^\infty(G;f)$ and $\epsilon >0$; we have, as required, $G x \se t U$ for all large enough scalars $t$, namely $t > \max_i \|f_i\|_\infty / \epsilon$.

For the cobounded type condition, observe that the topological dual of $E$ is $\ell^\infty(G;f)$ itself~\cite[3.14]{RudinFA}, endowed with the finest locally convex topology~\cite[p.~56]{Schaefer}, \cite[II \S\,6 No.\,1 Rem\,1)]{BourbakiEVT81}. Therefore the condition holds since by definition $\ell^\infty(G;f)$ is $G$-dominated by $f$.

At this point we apply the fixed-point property and obtain a positive element $\mu\neq 0$ in $C$. Upon renormalising, it remains only to show that $\mu(f)\neq 0$. This follows from the fact that $f$ is $G$-dominating  and $\mu$ is positive.\qed

\bigskip
We now prove the converse implication. Let $C$ be a weakly complete proper convex cone in a (locally convex, Hausdorff) topological vector space $E$. Suppose that we are given a $G$-representation of locally bounded cobounded type. We can therefore choose a non-zero point $x\in C$ with bounded $G$-orbit and a $G$-dominating element $\fhi\in E'$. To every $\lambda\in E'$ we associate a function $\hat\lambda$ on $G$ by $\hat\lambda(g) = \lambda(g x)$. This defines a linear positive map $\lambda\mapsto \hat\lambda$ from $E'$ to the space of functions on $G$ with the pointwise order. Notice that this map is $G$-equivariant for the usual $G$-actions by pre-composition; therefore its image is $G$-dominated by $\hat\fhi$. Moreover, since the orbit $Gx$ is bounded, $\hat\fhi$ is a bounded function. In conclusion, we have obtained a positive linear $G$-equivariant map from $E'$ to $\ell^\infty(G; \hat\fhi)$.

Let now $\mu$ be an integral on $G$ normalised for $\hat\fhi$. Composing the above map $\lambda\mapsto \hat\lambda$ with $\mu$, we obtain a $G$-invariant positive linear map $\check\mu$ on $E'$ with $\check\mu(\fhi)=\mu(\hat\fhi)=1$. A priori $\check\mu$ is an element of the algebraic dual ${E'}^*$ of $E'$. Suppose for a contradiction that $\check\mu$ is not in $C$ under the canonical embedding $C \se {E'}^*$. By weak completeness, $C$ is closed in ${E'}^*$ for the weak-* topology induced by $E'$. Therefore, the Hahn--Banach theorem implies that there is a continuous linear form $\lambda$ on ${E'}^*$ with $\lambda|_C \geq \alpha$ and $\lambda(\check\mu)< \alpha$ for some $\alpha\in \RR$. Since $C$ is a cone, we can suppose $\alpha = 0$. Moreover, by~\cite[3.14]{RudinFA}, we have $\lambda\in E'$ with $\lambda(\check\mu)=\check\mu(\lambda)$. Now $\lambda \geq 0$ implies that $\lambda(\check\mu)=\mu(\hat\lambda)$ must be positive since $\mu$ is positive and $\hat\lambda\geq 0$; this contradiction shows that $C$ contains indeed a non-zero fixed point, namely $\check\mu$.\qed

\begin{scholium}\label{sch:weaker:cb}
The above proof of \ref{pt:equiv:MG}$\Rightarrow$\ref{pt:equiv:FPC} for Theorem~\ref{thm:equiv} holds unchanged if we replace the assumption of cobounded type by the following weaker assumption on~$\fhi$: there exists some non-zero point $x\in C$ with bounded $G$-orbit such that $\fhi$ is $G$-dominating for the weaker pre-order on $E'$ given by the evaluation on the orbit $Gx$ (instead of the evaluation on all of $C$, which defines the dual pre-order). This minor extension of the fixed-point property for cones also follows immediately from Proposition~\ref{prop:min} below.
\end{scholium}

\section{Translate properties}\label{sec:ATP}
As we recalled in the introduction, the translate property for a bounded non-zero function $f\geq 0$ on $G$ is equivalent to the existence of an invariant integral on $\Span(G f)$. However, the more important question would be whether there is an invariant integral on the larger space $\ell^\infty(G; f)$. For instance, when $f=\one_A$ is the characteristic function of a subset $A\se G$, this stronger property is the one that is equivalent to the existence of an invariant finitely additive measure on $G$ normalised for $A$.

Unfortunately, to go from the translate property to this stronger one, it seems that it is necessary to know a priori that $G$ is amenable, see Problem p.~18 in~\cite{GreenleafBook}.

All this is exposed in~\cite[\S\,1]{GreenleafBook} and in~\cite{Rosenblatt74}, but perhaps a small measure of confusion has persisted. (For instance, the amenability assumption is \emph{not} used in Proposition~1.3.1 of~\cite{GreenleafBook}, though it is crucial in Theorem~1.3.2; in contrast, the amenability assumption is missing in Theorem~6.3.1 of~\cite{ShalomQI} for (2)$\Rightarrow$(3); finally, Corollary~1.2 of~\cite{Rosenblatt74} is \emph{not} a consequence, but rather a prerequisite, of the result preceding it.) We shall use notably a result of Moore~\cite{Moore_ramsey} to clarify the situation in Theorem~\ref{thm:TP-amen} below.

\medskip
We begin with a generalisation of the translate property.

\begin{defn}\label{def:ATP}
The group $G$ has the \textbf{abstract translate property} if the following holds.

For every $G$-representation on an ordered vector space $(V, \leq)$ and every $v\geq 0$ in $V$ for which there is a positive linear form $\delta\colon V\to\RR$ with $\delta(v) \neq 0$ which is bounded on the orbit $G v$, we have:
%
$$\sum_{i=1}^n  t_i g_i v \geq 0 \ \Longrightarrow\  \sum_{i=1}^n  t_i \geq 0\kern10mm(\forall t_i \in \RR, g_i\in G).$$
%
\end{defn}

Let us verify that the translate property is indeed a special case of the abstract translate property with $V=\ell^\infty(G)$. We just need to check that for any non-zero bounded function $f\geq 0$ on $G$, there is a positive linear form $\delta$ on $\ell^\infty(G)$ with $\delta(f) \neq 0$ and which is bounded on the orbit $G f$. Since $f\neq 0$, we can choose $g\in G$ with $f(g)\neq 0$. Then the evaluation at $g$ is a linear form $\delta$ with the desired properties.

The same verification shows also that the abstract translate property contains the case of the translate properties for $G$-sets considered by Rosenblatt~\cite{Rosenblatt_PhD} in the context of von Neumann's questions on finitely additive measures.

\begin{lem}\label{lem:TP-ATP}
The abstract translate property is equivalent to the translate property.
\end{lem}

\begin{proof}
We need to deduce the abstract translate property from the translate property. Let $V$, $v$ and $\delta$ be as in Definition~\ref{def:ATP}. Choose any $t_1, \ldots, t_n\in \RR$ and $g_1, \ldots g_n\in G$ with $\sum_{i=1}^n  t_i g_i v \geq 0$. Define a function $f$ on $G$ by $f(g)=\delta(g\inv v)$; observe that this definition is dual to the one made in the proof of~\ref{pt:equiv:MG}$\Rightarrow$\ref{pt:equiv:FPC} in Section~\ref{sec:MG}. Then $f$ is non-zero, bounded and non-negative. For any $g\in G$, we have
$$\sum_{i=1}^n  t_i g\inv g_i v  = g\inv \sum_{i=1}^n  t_i g_i v \geq 0.$$
Applying $\delta$ we find $\sum_{i=1}^n  t_i \delta( g\inv g_i v) \geq 0$ which means $\sum_{i=1}^n  t_i f(g_i\inv g) \geq 0$. Since this holds for all $g\in G$, we have $\sum_{i=1}^n  t_i g_i f \geq 0$ and hence the translate property yields $\sum_{i=1}^n  t_i \geq 0$ as desired.
\end{proof}

\begin{thm}\label{thm:TP-amen}
If a group satisfies the translate property, then it is amenable.
\end{thm}

\begin{proof}
By a result of Moore, it suffices to show that for every subset $E\se G$ there is a positive linear form $\mu$ on $\ell^\infty(G)$ normalised for $\one_G$ and such that $\mu(\one_E)=\mu(\one_{g E})$ for all $g\in G$, see Theorem~1.3 in~\cite{Moore_ramsey}.

Consider the subspace $D< \ell^\infty(G)$ spanned by all $\one_E - \one_{g E}$ as $g$ ranges over $G$. Denote by $Q$ the quotient space $\ell^\infty(G) \big/ \,\overline{\!D}$ with the quotient norm (where $\,\overline{\!D}$ denotes the norm-closure).

We claim that the $Q$-norm of the canonical image of $\one_G$ is one. For otherwise we can find $\epsilon>0$, $t_1, \ldots, t_n\in\RR$ and $g_1, \ldots, g_n\in G$ such that $v=\sum_i t_i (\one_E - \one_{g_i E})$ satisfies
$$\|\one_G - v\|_\infty \ \leq\ 1-\epsilon.$$
Thus $\one_G - v \leq (1-\epsilon)\one_G$ and hence $v\geq \epsilon\one_G$. In particular, $v - \epsilon \one_E$ is non-negative on $G$. However, such an element $v -  \epsilon\one_E$ stands in violation of the translate property because the sum of its coefficients is~$-\epsilon$. This proves the claim.

Now the Hahn--Banach theorem provides a continuous linear form $\tilde\mu$ of norm one on $Q$ such that $\tilde\mu({\one_G+\,\overline{\!D}})=1$. The corresponding linear form $\mu$ on $\ell^\infty(G)$ satisfies $\mu(\one_G)=1$ and is still of norm one; in particular, $\mu$ is positive. Since $\mu$ vanishes on $D$, it satisfies $\mu(\one_E)=\mu(\one_{g E})$ for all $g\in G$ as requested.
\end{proof}

The following consequence of Theorem~\ref{thm:TP-amen} completes the proof of Theorem~\ref{thm:equiv}, except for point~\ref{pt:equiv:Reiter} which we defer to Section~\ref{sec:Reiter}.

\begin{cor}\label{cor:TP-MG}
The group $G$ satisfies the translate property if and only if for every non-zero bounded function $f\geq 0$ on $G$, there is an integral on $G$ normalised for $f$.\qed
\end{cor}

The proof of Theorem~\ref{thm:TP-amen} given above uses only the fact that $G$ satisfies the translate property for characteristic functions of subsets of $G$. Therefore, we deduce the following variant of Corollary~\ref{cor:TP-MG}.

\begin{cor}\label{cor:TP-SM}
A group satisfies the translate property for sets if and only if it is supramenable.\qed
\end{cor}

\section{Smaller cones}\label{sec:smaller}
In preparation for the proof of Theorem~\ref{thm:stable}, we begin with two reduction statements.

\begin{prop}\label{prop:min}
In the definition of the fixed-point property for cones, one can assume in addition that the cone is minimal as an invariant closed convex cone containing a non-zero bounded orbit.

In particular, one can assume that the bounded orbits are dense in the cone.
\end{prop}

The reason why the minimality statement is not completely obvious is that the union of bounded orbits is generally not closed.

\begin{proof}[Proof of Proposition~\ref{prop:min}]
We consider a group $G$ with a representation on a weakly complete cone $C\se E$. We suppose that there is a non-zero element with a bounded $G$-orbit in $C$ and that $\fhi\in E'$ is $G$-dominating.

For the minimality, it suffices to show that for any chain $\sC$ of closed $G$-invariant convex cones $D\se C$ containing non-zero bounded orbits, the intersection $\bigcap \sC$ still contains a non-zero bounded orbit. Indeed, a linear form $\fhi\in E'$ that is $G$-dominating with respect to the pre-order dual to $C$ is a fortiori dominating for the pre-order induced on $E'$ by any sub-cone of $C$, and closed convex sub-cones are weakly closed and hence remain weakly complete.

Let thus $\sC$ be such a chain and choose in each $D\in \sC$ a point $x_D\neq 0$ whose orbit is bounded. Since $\fhi$ is $G$-dominating, there is $g\in G$ (depending on $D$) with $\fhi(g x_D) >0$. Since the orbit of $x_D$ is bounded, we can take $g$ such that moreover $\fhi(g x_D) \geq \frac12 \sup \fhi(G x_D)$. Therefore we can change our choice of $x_D$ by translating and renormalising, and assume that $\fhi(x_D) \geq 1/2$ and $\sup \fhi(G x_D)=1$ hold for all $D$.

Since $\fhi$ is $G$-dominating, it follows that for every $\lambda\in E'$, the value $\lambda(x_D)$ is bounded independently of $D$. Therefore, the net $(x_D)_{D\in \sC}$ has an accumulation point $x$ in the weak completion of $E$ and $x\in C$ by weak completeness of the cone. Since $\sC$ is a chain of weakly closed sets, $x\in\bigcap\sC$. Moreover, $x\neq 0$ because $\fhi(x)\geq 1/2$. Finally, the $G$-action being weakly continuous, we have $\fhi(g x)\leq 1$ for all $g\in G$; by $G$-domination, this implies that the orbit $G x$ is weakly bounded, and thus bounded by the generalised Banach--Steinhaus principle (Cor.~3 in~\cite[III \S\,5 No.\,3]{BourbakiEVT81} or~\cite[3.18]{RudinFA}).

The density statement follows from minimality since the set of points with bounded orbit is a $G$-invariant convex cone in $C$. 
\end{proof}

A different notion of smallness is provided by the following result.

\begin{prop}\label{prop:FPC:den}
For countable groups, the fixed-point property for cones holds as soon as it holds for all cones in spaces whose weak topology is first-countable.
\end{prop}

We cannot prove this simply by replacing $C$ with the closed convex cone spanned by a bounded $G$-orbit, because weakly separable spaces or cones are usually far from weakly first-countable ($L^2$ provides examples). It also seems that one cannot conduct a limit argument by exhausting $\ell^\infty(G;f)$ by subspaces of countable dimension. The reason is that subspaces of countable dimension cannot, in general, be \emph{full} in the sense of ordered vector spaces -- a fortiori not ideals in the Riesz space $\ell^\infty(G;f)$. This deprives us from the likes of the compactness argument invoked in Proposition~\ref{prop:min}.

\begin{proof}[Proof of Proposition~\ref{prop:FPC:den}]
Consider the argument given in Section~\ref{sec:MG} for the implication~\ref{pt:equiv:FPC}$\Rightarrow$\ref{pt:equiv:MG} of Theorem~\ref{thm:equiv}; we retain its notation.  Exactly the same argument can be applied if we replace the space $\ell^\infty(G;f)$ by the subspace $\Span(G f)$ spanned by the $G$-orbit of $f$. This leads to the a priori weaker conclusion that there is a $G$-invariant positive linear form on $\Span(G f)$, normalised for $f$. This implies the translate property, which we have proved to be in fact equivalent to the existence of an integral on $G$ normalised for $f$ (Corollary~\ref{cor:TP-MG} above). As we have already established in Section~\ref{sec:MG}, this in turn implies the fixed-point property for cones.

\smallskip
The reason why we travel this circuitous route is that when $G$ is countable, it is sufficient in this argument to test the fixed-point property for cones only in spaces $E$ whose weak topology admits a countable neighbourhood basis at each point. Indeed, we used $E=\Span(G f)^*$ and this topology is the weak-* topology induced by $\Span(G f)$, which is a space of countable algebraic dimension.
\end{proof}

\section{First stability properties}
We begin to address the stability properties of the class of groups with the fixed-point property for cones; closure under quotients is obvious. Thanks to Theorem~\ref{thm:equiv}, the closure under subgroups and directed unions becomes easy. Indeed, for subgroups it suffices to show the following.

\begin{lem}
Let $G$ be a group with the translate property. Then any subgroup $H<G$ has the translate property.
\end{lem}

\begin{proof}
If $f\geq 0$ is a non-zero bounded function on $H$, we consider it as a function on $G$ by extending it by zero outside $H$. Then the translate property for $f$ with respect to $G$ implies in particular the corresponding property with $H$.
\end{proof}

\begin{rem}\label{rem:sub:int}
Denote by $f\mapsto \tilde f$ the extension by zero outside $H$ in $G$; this yields a positive linear $H$-equivariant map $\ell^\infty(H; f) \to \ell^\infty(G; \tilde f)$. This shows that one can also establish the closure under subgroups using the characterisation in terms of integrals.
\end{rem}

As for directed unions, we only need the following.

\begin{lem}
Suppose that the group $G$ is the union of a family $\{G_\alpha\}_\alpha$ of subgroups which is directed with respect to inclusion. If each $G_\alpha$ has the abstract translate property, then $G$ has the translate property.
\end{lem}

\begin{proof}
Suppose that we have a counter-example to the translate property for $G$, namely $f\geq 0$ in $\ell^\infty(G)$, $t_1,\ldots, t_n \in \RR$ and $g_1, \ldots , g_n\in G$ with $\sum_i t_i g_i f \geq 0$ but $\sum_i  t_i <0$. Then this violates the \emph{abstract} translation property for any $G_\alpha$ containing all $g_i$ considered with its representation on $\ell^\infty(G)$. Such a $G_\alpha$ exists since the union is directed.
\end{proof}

As we have already mentioned, group extensions do not preserve the fixed-point property for cones. However, extensions by and of finite groups are covered by the next two lemmas.

\begin{lem}
Let $G$ be a group and $H<G$ a finite index subgroup. If $H$ has the fixed-point property for cones, then so does $G$.
\end{lem}

\begin{proof}
Consider a $G$-representation on a weakly complete convex proper cone $C\se E$ of locally bounded cobounded type. In order to apply the fixed-point property of $H$, we only need to justify that the $H$-action is of cobounded type. Let thus $\fhi\in E'$ be $G$-dominating. We fix a set $R\se G$ of representatives for the right cosets of $H$ in $G$ and define $\fhi_* = \sum_{r\in R} r\fhi$. Then $\fhi_*$ is $H$-dominating because any finite sum $\sum_i g_i \fhi$ with $g_i\in G$ can be written $\sum_i h_i r_i \fhi$ with $h_i\in H$ and $r_i\in R$, which is bounded by $\sum_i h_i  \fhi_*$ because $r_i\fhi \leq \fhi_*$ for all $i$. Therefore, $H$ fixes a point $x\neq 0$ in $C$. Now $x^*=\sum_{r\in R} r\inv x$ is a $G$-fixed point in $C$ and $x^*\neq 0$ because $C$ is a proper cone.
\end{proof}

\begin{lem}
Let $G$ be a group and $F\lhd G$ a finite normal subgroup. If $G/F$ has the fixed-point property for cones, then so does $G$.
\end{lem}

\begin{proof}
We consider again a $G$-representation on $C\se E$ as above. Since the subspace $E^F$ of $F$-fixed points is closed, it suffices to justify that we can apply the fixed-point property to the convex proper cone $C^F = C \cap E^F$ in the $G/F$-space $E^F$, this cone remaining weakly complete. This time the cobounded type is automatic because the map $E'\to (E^F)'$ is equivariant, positive and onto by Hahn--Banach. We need to check that $C^F$ has a non-zero point with bounded $G/F$-orbit (which is its $G$-orbit). Since $C$ is proper, the point $\sum_{g\in F} g x$ will do whenever $x\neq 0$ is a point of $C$ with bounded $G$-orbit. 
\end{proof}

\section{Central extensions}
In order to prove that any central extension of a group with the fixed-point property for cones retains this property, we can restrict ourselves to the case of countable groups thanks to points~\ref{pt:stable:sub} and~\ref{pt:stable:union} of Theorem~\ref{thm:stable}. Indeed, a subgroup of a central extension is a central extension of a subgroup. Our strategy is to prove that if a group $G$ has a representation of locally bounded cobounded type on a weakly complete convex proper cone $C\se E$ which is \emph{minimal} in the sense of Proposition~\ref{prop:min}, then the centre of $G$ acts trivially on $C$. This then establishes the stability under central extensions.

However, it will be crucial to know that we can in addition assume the weak topology of $E$ to have a countable neighbourhood basis, because this allows us to argue with extreme rays. Recall that even very familiar cones in separable Hilbert spaces need not have extreme rays (and their weak topology is not first-countable). This is the case of the self-dual cone in $L^2$ considered in Example~\ref{ex:L2}.

In order to ensure the countable neighbourhood basis, we apply Proposition~\ref{prop:min} to the cone provided by Proposition~\ref{prop:FPC:den}. We now choose an arbitrary element $z$ of the centre of $G$ and proceed to prove that $z$ fixes $C$ pointwise.

\medskip
The set $D=\{x + zx : x\in C\}$ is a convex $G$-invariant sub-cone of $C$. It is moreover closed. Indeed, if any net $(x_j + z x_j)_j$ with $x_j\in C$ converges, then so does the net $(x_j)$ upon passing to a subnet because the sum map $C\times C \to C$ is proper, see Cor.~1 in~\cite[II \S\,6 No.\,8]{BourbakiEVT81} together with Thm.~1(c) in~\cite[I \S\,10 No.\,2]{BourbakiTGI_alt}. This implies that the limit of $(x_j + z x_j)_j$ belongs to $D$. Furthermore, $D$ contains a non-zero bounded orbit because the map $x\mapsto x + zx$ is $G$-equivariant. Therefore, by minimality, $D=C$.

It follows that every extreme ray of $C$ is set-wise preserved by $z$. As was already mentioned, there need not exist any extreme ray in general; however, the fact that $E$ has a countable neighbourhood basis implies that $C$ is indeed the closed convex hull of its extreme rays, see Prop.~5 in~\cite[II \S\,7 No.\,2]{BourbakiEVT81}.

Thus, for $x\neq 0$ on an extreme ray, there is a positive scalar $\lambda_x$ such that $zx = \lambda_x x$. On the other hand, for any $\lambda>0$ we define closed convex $G$-invariant sub-cones $C_\lambda^{\pm}\se C$ by $C_\lambda^+ = \{x\in C : zx \geq \lambda x\}$ and $C_\lambda^- = \{x\in C : zx \leq \lambda x\}$. A priori we cannot play the minimality of $C$ against $C_\lambda^{\pm}$ because we do not know whether these cones contain non-zero bounded $G$-orbits. But $C_\lambda^+ + C_\lambda^-$ is still a convex $G$-invariant sub-cone and it is closed by Cor.~2 in~\cite[II \S\,6 No.\,8]{BourbakiEVT81}. Since every extreme ray lies in at least one of $C_\lambda^+$ or $C_\lambda^-$ (according to whether $\lambda_x\geq \lambda$ or $\lambda_x\leq \lambda$), we conclude that $C_\lambda^+ + C_\lambda^-$ is all of $C$.

We now claim that for each $\lambda>0$, one of the two cones $C_\lambda^{\pm}$ contains a non-zero bounded $G$-orbit. Let indeed $x\neq0$ be a point of $C$ with bounded $G$-orbit. Then there exists $x^\pm\in C_\lambda^\pm$ with $x=x^+ + x^-$. Using again the properness of the map $C\times C \to C$, we conclude readily that both $x^\pm$ have bounded $G$-orbits, whence the claim since and at least one of $x^\pm$ is non-zero.

\medskip
We can now invoke minimality again and deduce that for each $\lambda>0$, one of the two cones $C_\lambda^{\pm}$ coincides with $C$. Upon possibly replacing $z$ by $z\inv$, we can assume that for $\lambda=1$ we have $C^+_1 = C$. For all $\lambda > 1$, we must have $C_\lambda^- = C$ since otherwise $C_\lambda^+ = C$, which we claim is impossible. Indeed, there is a non-zero point $x\in C$ with bounded $G$-orbit, hence in particular bounded $\langle z\rangle$-orbit. But $z^n x\geq \lambda^n x$ allows us to apply the properness of $C\times C \to C$ to $z^n x = (z^n x -\lambda^n x ) + \lambda^n x$ and deduce that $\lambda^n x$ is bounded, contradicting $\lambda > 1$.

\medskip
In conclusion, $C= \bigcap_{\lambda>1} C_\lambda^-$ and $C=C^+_1$, which shows that $z$ acts trivially on $C$ -- as was to be proved.\qed

\begin{rem}\label{rem:coif}
An alternative to Proposition~\ref{prop:FPC:den} to ensure the existence of (enough) extreme rays when $G$ is countable is to appeal to Choquet's notion of cones that are \emph{bien coiff\' es} (approximately translated to ``well-capped''). This is a weaker property than first countability and it can be established for cones associated to $\ell^\infty(G;f)^*$ with $G$ countable. The method consists in exhausting $\ell^\infty(G;f)$ by order ideals associated to translates of $f$; although these ideals are not of countable dimension, they give rise each to a cone that can be shown to be bien coiff\' e. This yields a projective limit of a countable system of (non-invariant) cones which remains bien coiff\'e by Ex.~29(b) in~\cite[II \S\,7]{BourbakiEVT81} -- cf.\ Thm.~23 in~\cite{Choquet63}.
\end{rem}

\begin{scholium}
A variation on a subset of the arguments in the above proof of Theorem~\ref{thm:stable}\ref{pt:stable:central} yields the following:

Consider a $G$-representation of locally bounded cobounded type on any weakly complete proper convex cone $C$. Let $Z\colon E\to E$ be any continuous linear operator which commutes with $G$ and preserves $C$. If $C$ is bien coiff\'e, then there is a closed convex $G$-invariant sub-cone $C' \se C$ which contains a non-zero bounded $G$-orbit and such that $Z$ acts by scalar multiplication on $C'$. Indeed, we find first a sub-cone $C'$ that is minimal only within the smaller class of cones $D$ satisfying in addition $Z D\se D$. We then argue that $(\Id+Z) C'$ must be closed in $C'$ and deduce $(\Id+Z) C'=C'$. In particular $Z$ preserves all extreme rays of $C'$. Since $C'$ is also bien coiff\'e, it is the closed convex hull of its extreme rays. An argument with subsets $C^\pm_\lambda$ as above implies that for some $\lambda\geq 0$, the operator $Z$ coincides with $\lambda\Id$ on all of $C'$. At this point we see that $C'$ is moreover minimal in the sense of Proposition~\ref{prop:min} (regardless of $Z$).
\end{scholium}

\section{Subexponential groups}
Recall that a group $G$ is called \textbf{subexponential} if $\lim_{n\to\infty} |S^n|^{1/n}=1$ for every finite subset $S\se G$, where $S^n$ is the set of $n$-fold products of elements of $S$. This holds in particular for all locally finite or abelian groups.

The following property was first recorded by Jenkins, who proved moreover that it characterises subexponentiality~\cite[Lemma~1]{Jenkins76}.

\begin{lem}\label{lem:SE}
Suppose $G$ subexponential. For every finite subset $S\se G$ and every $\epsilon>0$ there is $\ro\neq 0$ in $\ell^1(G)$ with $|s\ro - \ro| \leq \epsilon \ro$ for all $s\in S$.
\end{lem}

\begin{proof}
Choose $\delta>0$ such that $|e^\delta -1|\leq \epsilon$ and $|e^{-\delta} -1|\leq \epsilon$. Define the function $\ro$ on $G$ by $\ro(g) = e^{- \delta |g|_S}$, where $|g|_S$ is the (symmetrised) word distance associated to $S$, with the convention $\ro(g)=0$ if $g$ is not in the subgroup generated by $S$. The desired inequality holds by construction and $\ro$ is summable because  $G$ is subexponential.
\end{proof}

Lemma~\ref{lem:SE} can of course also be reformulated with the inequality $|\ro s - \ro| \leq \epsilon \ro$; it suffices to switch left and right multiplication by replacing $S$ with $S\inv$ and $\ro$ with $h\mapsto \ro(h\inv)$. In fact, the function constructed in the above proof satisfies both versions of this inequality simultaneously.

\bigskip
The following proposition will complete the proof of the last remaining item in Theorem~\ref{thm:stable}.

\begin{prop}
Let $G$ be a group with the abstract translate property and let $H$ be a subexponential group. Then ${G\times H}$ has the translate property. 
\end{prop}

\begin{proof}
Let $f\geq 0$ be a non-zero bounded function on ${G\times H}$ and suppose for a contradiction that we have $g_i\in G$, $h_i\in H$ and $t_i\in \RR$ with $\sum_{i=1}^n  t_i g_i h_i f \geq 0$ but $\sum_{i=1}^n  t_i <0$. Then we can choose $0<\epsilon<1$ small enough to ensure that $\sum_{i=1}^n  \tilde t_i <0$ holds for the numbers $\tilde t_i \geq t_i$ defined by 
$$\tilde t_i = \begin{cases}%
(1+\epsilon)t_i &\mbox{if } t_i>0, \\
(1-\epsilon)t_i & \mbox{if } t_i\leq 0. \end{cases}$$
We apply the right-handed version of Lemma~\ref{lem:SE} to $H$ with this $\epsilon$ and the set $S=\{h_1, \ldots, h_n \}$. This yields $\ro\neq 0$  in $\ell^1(H)$ with $|\ro h_i - \ro| \leq \epsilon \ro$ for all $i$. The usual convolution of $\ell^1(H)$ against $\ell^\infty({G\times H})$ obtained by viewing $H$ as a subgroup of ${G\times H}$ is
$$(\ro f)(x) = \sum_{y\in H} \ro(y) f(y\inv x)\kern10mm (x\in G\times H).$$
Furthermore, $(\ro h_i) f = \ro (h_i f)$ holds and $\ro f$ is still a non-zero non-negative bounded function on ${G\times H}$. 

\smallskip
Since $f\geq 0$ and $(1-\epsilon) \ro \leq \ro h_i  \leq (1+\epsilon) \ro$, we deduce
$$(1-\epsilon) \ro f \leq \ro h_i f \leq (1+\epsilon) \ro f \kern10mm(\forall i).$$
Multiplying this by $t_i$ and taking into account the sign of $t_i$, we deduce $\tilde t_i \ro f \geq t_i  \ro h_i f$ for all $i$.

\medskip
We shall now apply the abstract translate property of $G$ to the element $\ro f$ of $\ell^\infty({G\times H})$ for $g_i$ and $\tilde t_i$. For the linear form $\delta$ required by Definition~\ref{def:ATP}, we can take the evaluation at any point of ${G\times H}$ where $f$ does not vanish. Since $\ro$ and $g_i$ commute, we have
$$\sum_i \tilde t_i g_i \ro f \geq \sum_i t_i g_i \ro h_i f =  \ro \sum_i t_i g_i h_i f \geq 0.$$
This, however, contradicts $\sum_{i=1}^n  \tilde t_i <0$.
\end{proof}

\section{Reiter and ratio properties}\label{sec:Reiter}
This section regards approximation properties characterising the existence of integrals on $\ell^\infty(G; f)$ and on $\Span(G f)$. It is a rather direct extension of the results known in the special case $f=\one_A$ of characteristic functions of subsets.

\medskip
The following definition is a generalisation to functions of the \emph{ratio property} introduced in~\cite[\S\,5]{Rosenblatt_PhD} and~\cite[\S\,2]{Rosenblatt73} for subsets.

\begin{defn}\label{def:ratio}
Given a function $f$ on a group $G$, we say that $G$ has the \textbf{$f$-ratio property} if for every finite set $S\se G$ and every $\epsilon>0$ there is a finitely supported function $u$ on $G$ with
$$\| (s u)\cdot f\|_1 \ <\ (1+\epsilon)\| u\cdot f\|_1 \kern10mm \forall\, s\in S.$$
Here $\|\cdot\|_1$ denotes the $\ell^1$-norm and $u\cdot f$ is the point-wise product whilst $s u$ is as usual the translation of $u$ by $s$. We will only consider the case where $f$ is bounded.
\end{defn}

\begin{rem}\label{rem:ratio}
This definition can only be satisfied if $f\neq 0$, and we have necessarily $u\neq 0$. Moreover, the definition is unchanged if we replace $f$ by $|f|$ and we can always assume $u \geq 0$.
\end{rem}

Next, still extending~\cite{Rosenblatt_PhD}, \cite{Rosenblatt73}, we consider a generalisation of the group property introduced by Dieudon\-n{\'e}~\cite[p.~284]{Dieudonne60} and now called the \emph{Reiter property}.

\begin{defn}\label{def:Reiter}
$G$ has the \textbf{$f$-Reiter property} if for every finite set $S\se G$ and every $\epsilon>0$ there is a finitely supported function $u$ on $G$ with
$$\| (s u - u)\cdot f\|_1 \ <\ \epsilon\| u\cdot f\|_1 \kern10mm \forall\, s\in S.$$
The same comments as in Remark~\ref{rem:ratio} hold.
\end{defn}

Just as the Reiter property for groups characterises amenability~\cite[Ch.\,8 \S\,6]{Reiter68}, we have:

\begin{prop}\label{prop:Reiter}
Let  $f\geq 0$ be a  non-zero bounded function on a group $G$. Then there is an integral on $G$ normalised for $f$ if and only if  $G$ has the $f$-Reiter property.
\end{prop}

In particular, this establishes the remaining equivalence~\ref{pt:equiv:Reiter} of Theorem~\ref{thm:equiv}.

\begin{rem}
Rosenblatt also proposes a variant where $u$ is the characteristic function of a finite set, extending F{\o}lner's theorem~\cite[\S\,2]{Folner}. It is not clear that this should be possible in the current context where $f$ is a function rather than a subset of $G$.
\end{rem}

\medskip
It is apparent on the definitions that the $f$-Reiter property implies the $f$-ratio property; this is consistent with the following since the existence of integrals implies the translate property.

\begin{prop}\label{prop:ratio}
Let  $f\geq 0$ be a  non-zero bounded function on a group $G$. Then $f$ satisfies the translate property if and only if  $G$ has the $f$-ratio property.
\end{prop}

In conclusion, Theorem~\ref{thm:equiv} shows that the fixed-point property for cones can be characterised both by the $f$-ratio property and by the $f$-Reiter property, letting $f$ range over all non-zero positive bounded functions on $G$.

\begin{proof}[Proof of Proposition~\ref{prop:ratio}]
We follows ideas of Rosenblatt~\cite{Rosenblatt_PhD}. The $f$-ratio property implies that there is an integral on $\Span(G f)$ by taking an accumulation point of the net
$$\frac{\sum_{g\in G} u(g) h(g)}{\| u\cdot f\|_1 },\kern 10mm h\in \Span(G f)$$
indexed by $(S, \epsilon)$, with $u\geq 0$ as in Definition~\ref{def:ratio}. Such an accumulation point $\mu$ exists since for any fixed $h$ the above net of scalars is bounded for cofinal indices $(S, \epsilon)$ depending on $h$. For any $g$, we have $\mu( g f) = \mu(f)$ by construction; it follows that $\mu$ is indeed invariant on  $\Span(G f)$.

For the converse, we adapt to our setting a construction attributed to Kakutani in~\cite[p.~15]{Rosenblatt_PhD}. Suppose that $f$ has the translate property and fix $(S, \epsilon)$; we can assume that $S$ contains the identity. Consider the element $v=(s\inv f)_{s\in S}$ of $\ell^\infty (G, \RR^S)$ and denote by $P\se (\RR_+)^S$ the closure of the convex cone generated by the image of $v$ in $\RR^S$ (in~\cite{Rosenblatt_PhD}, the cone was already closed). The translate property applied to $S$ (i.e.\ with the elements $g_i$ of Definition~\ref{def:TP} enumerating $S$) is equivalent, by Hahn--Banach, to the statement that $\one_S\in\RR^S$ belongs to $P$.

Therefore, we can choose a finite set $\{x_1, \ldots, x_m\}\se G$ and coefficients $c_j\geq 0$ such that $\big(\sum_{j=1}^m c_j (s\inv f)(x_j)\big)_{s\in S}$ is arbitrarily close to $\one_S$. Specifically, we make the choice so that for all $s\in S$ we have $\big| 1- \sum_{j=1}^m c_j (s\inv f)(x_j) \big| <\epsilon/(2+\epsilon)$. Defining $u = \sum_{j} c_j \delta_{x_j}$, this reads as ${\big|1- \| (s u)\cdot f\|_1 \big|<\epsilon}/(2+\epsilon)$. Since $e\in S$, we deduce
$$\| (s u)\cdot f\|_1 \ <\ 1 + \frac{\epsilon}{2+\epsilon} \ =\  (1+\epsilon) \left(1- \frac{\epsilon}{2+\epsilon}\right)\ <\ (1+\epsilon)\| u\cdot f\|_1$$
as in the definition of the $f$-ratio property.
\end{proof}

\begin{proof}[Proof of Proposition~\ref{prop:Reiter}]
The $f$-Reiter property implies the existence of an integral on $\ell^\infty(G;f)$ as in the above proof of Proposition~\ref{prop:ratio}, noting that the stronger condition in Definition~\ref{def:Reiter} ensures the invariance of the accumulation point for all $h\in \ell^\infty(G;f)$ rather than just $h\in \Span (G f)$.

The converse can be proved as in~\cite{Rosenblatt73}. More precisely, the existence of an integral normalised for $f$ implies that there is a net $(u_\alpha)$ of finitely supported functions $u_\alpha\geq 0$ on $G$ with $u_\alpha \cdot f \neq 0$ and such that
$$\frac{\sum_{g\in G} (s u_\alpha)(g) h(g) - u_\alpha(g) h(g)}{\| u_\alpha\cdot f\|_1 }\longrightarrow 0\kern 10mm \forall\,h\in \ell^\infty(G;f)$$
as in Theorem~3.3 of~\cite{Rosenblatt73}, noting that~\cite[Prop.~2.2]{Rosenblatt73} is established in the necessary generality. Going from this weak convergence to the convergence in norm of Definition~\ref{def:Reiter} requires an appropriate version of the Mazur trick, established in this generality in Lemma~4.8 of~\cite{Rosenblatt73}.
\end{proof}

\section{Further comments and questions}
\subsection{Topological groups}
The definition of the fixed-point property for cones can be extended to topological groups simply by adding the requirement that the representation be (orbitally) continuous. For applications to integrals, the corresponding change is to consider those bounded functions that are \emph{right uniformly continuous}. This topological generalisation can be pursued in two directions:

\smallskip
On the one hand, the case of locally compact groups. Here, most arguments can be adapted with suitable technical precautions. For instance, Jenkins' criterion for subexponential groups (Lemma~1 in~\cite{Jenkins76}) was stated and proved in this generality.

\smallskip
On the other hand, for more general topological groups, the corresponding equivalent characterisations are less flexible and we cannot prove the stability properties of Theorem~\ref{thm:stable} anymore. It turns out that such statements \emph{do indeed fail}; this does not, of course, make the fixed-point property for cones any less intriguing for these groups.

\begin{example}\label{ex:sym}
Consider the Polish topological group $\Sym$ of all permutations of a countable set (with the pointwise topology). We claim that $\Sym$ has the fixed-point property for cones as a topological group, although it contains non-abelian free groups as closed subgroups.

To justify the claim, it suffices to prove that this  fixed-point property is inherited from dense subgroups, since $\Sym$ contains as a dense subgroup the group of finitely supported permutations, which has the fixed-point property for cones even as an abstract group, being a directed union of finite groups. The only non-trivial point is the following. Consider an orbitally continuous representation of a topological group $G$ on a cone $C\se E$ as usual and assume that it is of locally bounded cobounded type. We need to show that the action of any dense subgroup $H$ remains of cobounded type, or at least satisfies the weaker condition of Scholium~\ref{sch:weaker:cb}.

Let thus $\fhi\in E'$ be $G$-dominating and fix some non-zero point $x\in C$ with bounded $G$-orbit. Pick any $\lambda\in E'$. For any given $0\neq y\in C$, the cobounded type assumption for $G$ implies that there are $g_1,\ldots, g_n\in G$ such that $\lambda (y) < \sum_i \fhi(g_i y)$ (adding, if needed, some $g$ with $\fhi(g y) > 0$, which exists by domination). Now the sequence $(g_i)$ depends a priori on $y$ because we want this \emph{strict} inequality, but this allows us to choose these $g_i$ in $H$ by density and orbital continuity. On the other hand, for a given sequence $g_1,\ldots, g_n\in H$, this strict inequality defines a weakly open set in $C$. Therefore, by relative weak compactness of the orbit $H x$, we can concatenate finitely many such sequences to verify the condition of Scholium~\ref{sch:weaker:cb} for $H$.
\end{example}

The above example is analogous to the case of amenability of topological groups that are not locally compact, see~\cite[p.~489]{Harpe82s}.

\subsection{Relativising}
Given a subgroup $H<G$ of a group $G$, we say that $G$ has the fixed-point property for cones \textbf{conditionally} with respect to $H$ if the following holds: for every $G$-representation of cobounded type on any weakly complete proper convex cone that admits a non-zero $H$-fixed point with bounded $G$-orbit, there is a non-zero $G$-fixed point. Thus, the condition of \emph{locally bounded type} has been conditioned on $H$.

\smallskip
It is straightforward to check that for \emph{normal} subgroups $H\lhd G$, this conditional property is equivalent to  the fixed-point property for cones for the quotient group $G/H$. In that sense, the conditional property is the analogue of the \emph{co-amenability} studied notably by Eymard~\cite{Eymard72}.

\medskip
These conditional properties offer a non-trivial generalisation of the corresponding absolute fixed-point properties when we consider  \emph{non-normal} subgroups $H<G$. For instance, in the case of amenability, it was shown that co-amenability does not pass to subgroups in the natural sense~\cite{Monod-Popa}, \cite{Pestov}. In the present case, the situation is different:

\begin{prop}
Consider groups $K<H<G$. If $G$ has the fixed-point property for cones conditionally with respect to $K$, then so does $H$.
\end{prop}

This proposition is a consequence of the following one thanks to Remark~\ref{rem:sub:int}.

\begin{prop}
A group $G$ has the fixed-point property for cones conditionally with respect to $H<G$ if an only if there exists an integral on $G/H$ normalised for $f$ whenever $f\geq 0$ is a non-zero bounded function on $G/H$.
\end{prop}

The latter proposition is proved by the arguments of Section~\ref{sec:MG} without any change.

\medskip
Contrariwise, the conditional property for cones lacks a feature of co-amenability, namely it is not transitive. This is due to the fact that the fixed-point property for cones is not preserved by group extensions.

\medskip
The conditional property can be further extended to group actions, generalising the $G$-action on $G/H$. In that setting, we simply replace the locally bounded type assumption by the following condition for a $G$-action on a set $X$: we assume that there be a non-zero $G$-map from $X$ to the cone whose image is bounded. The corresponding fixed-point property can again be characterised in terms of integrals on $X$ by the arguments of Section~\ref{sec:MG}.

\medskip
Finally, we point that there is a dual way to relativise fixed-point properties to subgroups, which in the case at hand leads to the following. A subgroup $H<G$ has the fixed-point property for cones \textbf{relatively to $G$} if for every $G$-representation of locally bounded type and of $H$-cobounded type on any weakly complete proper convex cone, there is an $H$-fixed point.

In the classical case of amenability, this corresponds to \emph{relative amenability} as studied in~\cite{Caprace-Monod_rel}. This notion turns out to occur naturally in the study of Zimmer-amenability of non-singular actions on measure spaces, of group C*-algebras and boundaries~\cite{Ozawa_Cst}, as well as of invariant random subgroups~\cite{Bader-Duchesne-Lecureux}.

Just as in the case of relative amenability, this second relativisation may seem empty at first sight:

\begin{prop}
For abstract groups $H<G$, the fixed-point property for cones relatively to $G$ is equivalent to the fixed-point property for cones of $H$.
\end{prop}

Indeed the property for $H$ is formally stronger and the converse follows from Remark~\ref{rem:sub:int}.

\medskip
However, in the topological context, the relative property is strictly weaker since it follows from the the fixed-point property for cones of the larger group $G$; thus, Example~\ref{ex:sym} above is a case where these properties are not equivalent.

\subsection{Review of some other properties}\label{sec:other}
Conical fixed point properties have been considered earlier, but mostly, as far as we know, for the special case of cones with a compact base.

In that setting, Furstenberg~\cite[\S\,4]{Furstenberg65} introduced the notion of \emph{Tychonoff groups} by requiring that a ray be preserved (not necessarily point-wise fixed). Since the cone has a compact base $K$, this is equivalent to fixing a point in the projective action on the compact convex set $K$. This property was further studied by Conze--Guivarch~\cite{Conze-Guivarch} and Grigorchuk~\cite{Grigorchuk98}. Albeit very interesting, the Tychonoff property is not relevant to the questions we investigated here (in particular the von Neumann--Tarski invariant measure problem). Indeed, already the infinite dihedral group -- which is of polynomial growth -- fails to be Tychonoff. This is witnessed e.g.\ by its representation as matrix group generated by $\begin{pmatrix}2&0\\0&1/2\end{pmatrix}$ and $\begin{pmatrix}0&1\\1&0\end{pmatrix}$ acting on the positive quadrant of $\RR^2$. Conversely, but in the locally compact generality,  connected groups of upper-triangular matrices are Tychonoff~\cite[p.~284]{Furstenberg65}, but fail to be supramenable unless they have polynomial growth~\cite{Jenkins73}.

\medskip
Still for cones with a compact base, Grigorchuk~\cite[\S\,5]{Grigorchuk98} considered the non-zero fixed-point property under the assumption that (all) orbits be bounded -- this property follows from the fixed-point property introduced in the present article. He proved that \emph{Liouville groups} enjoy that property for cones with compact bases, if we define ``Liouville'' by the triviality of the Poisson boundary of some \emph{finitely supported} generating symmetric random walk. Liouville groups include all finitely generated subexponential groups by a result of Avez~\cite{Avez74}, and we can reduce to the finitely generated case by a compactness argument. But Liouville groups also include some groups with free non-abelian sub-semigroups, for instance the lamplighter $\ZZ/2 \wr \ZZ$ (see e.g.\ Thm.~1.3 in~\cite{Kaimanovich83s}); such groups have paradoxical subsets.

\bigskip
Jenkins has introduced in~\cite{Jenkins80} a non-zero fixed-point property for certain compact convex sets that are allowed to be more complicated than initial slices of cones with a compact base. He proved that his fixed-point property implies the existence of integrals (Proposition~5 in~\cite{Jenkins80}). However, it does not hold for all subexponential groups (Remark p.~370 in~\cite{Jenkins80}) and therefore is too restrictive for our purposes.

\bigskip
Finally, another property introduced earlier by Jenkins~\cite{Jenkins76} is his ``property~F'' (not to be confused with a different ``property~F'' that he introduced earlier yet~\cite{Jenkins74}). We recall his definition, which is related to our abstract translate property:

\itshape
Suppose that $G$ has a linear representation on a vector space $V$ and that we are given $x\in V$ and $\fhi\in V^*$ (the algebraic dual) such that the function $g\mapsto \fhi(gx)$ is non-zero, bounded and~$\geq 0$ on $G$. Then there exists a net $(\fhi_j)_j$ of finite positive linear combinations of $G$-translates of $\fhi$ such that, for any given $g\in G$, the net $(\fhi_j(g x))_j$ converges to~$1$.\upshape

\smallskip
Jenkins proved~\cite[Thm.~2]{Jenkins76} that subexponential groups enjoy this property~F. We claim that in fact this property is equivalent to our fixed-point property for cones. It is rather straightforward to see that property~F implies the abstract translate property. Now the latter implies the fixed-point property for cones by Theorem~\ref{thm:equiv} above. It remains to see that the fixed-point property for cones implies property~F. To this end, consider $\Span(Gx)<V$ in the above notation, and endow its algebraic dual $E=\Span(Gx)^*$ with the weak-* topology. Consider the cone of linear forms that are~$\geq 0$ on the whole orbit $Gx$. Then an argument entirely similar to that of Section~\ref{sec:MG} implies the existence of a net as required by property~F.

\medskip
This equivalence also clarifies the picture with respect to another result of Jenkins: he proved~\cite[Cor.~3]{Jenkins76} that subexponential groups, in addition to having property~F, also admit integrals. It follows from the above equivalence (and again Theorem~\ref{thm:equiv}) that the existence of integrals is after all equivalent to property~F.

\subsection{Questions and problems}
As we recalled in the introduction, amenability is equivalent to the fixed-point property for general convex compact sets in arbitrary (Hausdorff) locally convex spaces. However, when the group is countable, it suffices to consider \emph{Hilbert spaces} (in their weak topology) instead of general locally convex spaces; this was recently proved in~\cite{Gheysens-Monod_pre}.

\begin{question}
Is it sufficient to check the conical fixed-point property on weakly complete cones in Hilbert spaces? How about weakly locally compact cones in Hilbert spaces?
\end{question}

The corresponding representations would of course not be isometric in view of Theorem~\ref{thm:RNC} (for countable groups). Also, the second question is not suited for general (uncountable) groups: for instance, a group with Bergman's property~\cite{Bergman06} would necessarily acts equicontinuously (compare~\cite{Gheysens-Monod_pre}) and hence would have a non-zero fixed point by Theorem~\ref{thm:RNCD}, although it might not even be amenable; an example is $\Sym$ viewed as a group without topology.

\bigskip
It would be desirable to have a wider menagerie of examples of groups distinguishing the properties that we consider. First in line is the following:

\begin{problem}
Find a finitely generated group of exponential growth enjoying the fixed-point property for cones.
\end{problem}

In view of Theorem~\ref{thm:stable}, such an example would in particular be provided by solving another problem:

\begin{problem}
Exhibit two groups with the fixed-point property for cones such that their product does not have this property.
\end{problem}

We recall that the corresponding problem for supramenable groups is also open; in view of~\cite[3.4]{Kellerhals-Monod-Rordam}, the latter problem is equivalent to:

\begin{problem}
Give an example of two groups $G_1$, $G_2$ such that $G_1\times G_2$ admits a Lipschitz embedding of the binary tree but neither of the $G_i$ does.
\end{problem}

Comparing the last two problems, we record that we do not know if the implication from the fixed-point property for cones to supramenability can be reversed:

\begin{question}
Is supramenability equivalent to the fixed-point property for cones?
\end{question}

The classical Reiter property has equivalent $L^p$ variants (already introduced in~\cite[p.~284]{Dieudonne60}). Likewise, the group $G$ satisfies the $f$-ratio property for all non-zero $f\in\ell^\infty(G)$ if and only if the same holds with the $\ell^1$-norm replaced by any $\ell^p$-norm with $1\leq p < \infty$. This is a straightforward verification where $f$, $u$ and $\epsilon$ are to be replaced with appropriate powers. The same is true for the $f$-Reiter property using basic $L^p$ inequalities.

\medskip
The particular case of $p=2$ suggests the following:

\smallskip
We say that a unitary representation of a group $G$ on a Hilbert space $V$ has the \textbf{operator $\alpha$-ratio property} for a continuous operator $\alpha\neq 0$ of $V$ if for every finite set $S\se G$ and every $\epsilon>0$, there is $v\in V$ with $\| \alpha(s v)\| <(1+\epsilon)\| \alpha( v) \|$ for all $s\in S$. Likewise, the representation has the stronger \textbf{operator $\alpha$-Reiter property} if we can choose $v$ with $\| \alpha(s v - v)\| <\epsilon\| \alpha( v) \|$ for all $s\in S$. 

\begin{problem}
Clarify under which circumstances the operator $\alpha$-ratio and $\alpha$-Reiter properties hold.
\end{problem}

Returning closer to supramenability:

\begin{problem}
Characterise the unitary representations that satisfy the operator $\alpha$-ratio or $\alpha$-Reiter property for all non-zero \emph{projections} $\alpha$ in $V$.
\end{problem}


\bibliographystyle{acm}
\bibliography{../BIB/ma_bib}

\end{document}